\documentclass[11pt,reqno,a4paper]{amsart}

\usepackage{todonotes}
\usepackage{amsmath}
\usepackage{amsfonts}
\usepackage{mathrsfs}  
\usepackage{amssymb}
\usepackage{amsthm}
\usepackage{hyperref}
\usepackage{graphicx}
\usepackage[noadjust]{cite}
\usepackage{caption}
\usepackage{dutchcal}
\usepackage{bbm}
\usepackage{tikz}
\usetikzlibrary{decorations.pathreplacing}
\usepackage[T1]{fontenc}
\usepackage[utf8]{inputenc}
\usepackage{mathtools}

\usepackage[top=2.9cm, bottom=2.9cm, left=2.5cm, right=2.5cm, headsep=0.2in]{geometry}

\usepackage{fancyhdr}

\newtheorem{theorem}{Theorem}[section]
\newtheorem{remark}[theorem]{ Remark}

\newtheorem{proposition}[theorem]{Proposition}

\newtheorem{lemma}[theorem]{Lemma}

\newtheorem{definition}[theorem]{Definition}

\newcommand{\V}{\Vert}
\newcommand{\RR} {\mathbb R}
\newcommand{\CC} {\mathbb C}

\newcommand{\inrad}{\operatorname{inrad}}
\newcommand{\locnod}{\Omega_{\loc}}

\newcommand{\pa} {\partial}

\newcommand{\beq} {\begin{equation}}
\newcommand{\eeq} {\end{equation}}

\newcommand{\loc} {\operatorname{loc}}

\newcommand{\Ric} {\operatorname{Ric}}

\newcommand{\diam}{\operatorname{diam}}
\newcommand{\capacity}{\operatorname{cap}}

\renewcommand{\Re} {\operatorname{Re}}

\begin{document}
\title[Mass concentration, avoided crossings and (sub)-level sets]{Some applications of heat flow to Laplace eigenfunctions}

\author{Bogdan Georgiev and Mayukh Mukherjee}
\address{Max-Planck Institute for Mathematics, Vivatsgasse 7, 53111, Bonn, Germany}
\curraddr{Fraunhofer Institute, IAIS, Schloss Birlinghoven, 53757, Sankt Augustin, Germany}
\email{bogdan.georgiev@iais.fraunhofer.de}

\address{Indian Institute of Technology Bombay\\ Powai, Maharashtra - 400076 
	\\ India}

\email{mukherjee@math.iitb.ac.in, mathmukherjee@gmail.com}

\begin{abstract}

We consider mass concentration properties of Laplace eigenfunctions $\varphi_\lambda$, that is, smooth functions satisfying 
the equation $-\Delta \varphi_\lambda = \lambda \varphi_\lambda$, on a smooth closed Riemannian manifold. Using a heat diffusion technique, we first discuss mass concentration/localization 
properties of eigenfunctions around their nodal sets.  Second, we discuss the problem of avoided crossings and (non)existence of nodal domains which continue to be thin over relatively long distances. 
Further, using the above techniques, we discuss the decay of Laplace eigenfunctions on Euclidean domains which have a central ``thick'' part and ``thin'' elongated branches representing  tunnels of sub-wavelength opening. 
Finally, in an Appendix, we 
record some new observations regarding sub-level sets of the eigenfunctions and interactions of 
different level sets.


\end{abstract}
\maketitle
\section{Introduction}

We consider a closed $n$-dimensional Riemannian manifold $M$ with smooth metric $g$, and the Laplacian $-\Delta$ on $M$ (we use the analyst's sign convention, namely, $-\Delta$ is positive semidefinite). It is known that in this setting $-\Delta$ has discrete spectrum $0 = \lambda_1 < \lambda_2 \leq \dots \leq \lambda_k \nearrow \infty$. We are interested in the behaviour of the {\em high energy} eigenfunctions
\begin{equation}
	-\Delta \varphi_\lambda = \lambda \varphi_\lambda, 
\end{equation}
that is, eigenfunctions for large eigenvalues $ \lambda $. 

A motivational perspective is given, for instance, through quantum mechanics where the $ L^2 $-normalized eigenfunctions induce a probability density $ \varphi_\lambda^2(x) dx $, i.e., the probability density of a particle of energy $ \lambda $ to be at $ x \in M $. In this direction, questions concerning the distribution and geometry of $ \varphi_\lambda $ emerge naturally: these include estimates on $ L^p $-norms,  concentration/equidistribution phenomena, bounds on the zero (nodal) sets and smallness, etc. For a far reaching and accessible overview, which includes many classical as well as more recent results, we refer to \cite{S}. Also see \cite{Z}, \cite{Z1}.

We first fix some definitions/notations.
    For an eigenvalue $\lambda$ of $-\Delta$ and a corresponding eigenfunction $\varphi_\lambda$, we denote the set of zeros (nodal set) of $\varphi_\lambda$ by $N_{\varphi_\lambda} := 
	\{ x \in M : \varphi_\lambda (x) = 0\}$. Given a nodal set $ N_{\varphi_\lambda} $ we call the connected components of $ M \setminus N_{\varphi_\lambda} $ nodal domains. As notation for a given nodal domain we use $ \Omega_\lambda $, or just $\Omega$ with slight abuse of notation. Further, we denote the (metric) tubular neighbourhood of width $\delta$ around the nodal set $N_{\varphi_\lambda}$ by $T_\delta$. Throughout the paper, $|S|$ and $\pa S$ denote the volume and the boundary of the set $S$ respectively. The letters $c, C$ etc. are used to denote constants dependent on $(M, g)$ and independent of $\lambda$. The values of 
	$c, C$ can vary from line to line. 
	When two quantities $X$ and $Y$ satisfy $X \leq c_1 Y$ and $X \geq c_2Y$, we write $X \lesssim Y$ and $X \gtrsim Y$ respectively. When both are satisfied, 
	we write $X \sim Y$ in short. 
	Normally, 
	our estimates will be up to constants which might be dependent on the geometry of the manifold $(M, g)$, but definitely not on the 
	eigenvalue $\lambda$.
	
Given a point $y \in \RR^n$ and a set $S \subset \RR^n$, let $\psi_S(t, y)$ denote the probability that a Brownian particle starting at $y$ ends up inside $S$ within time $t$. For the particular case of $S = B(y, r)$, we observe that $\psi_{\RR^n \setminus B(y, r)}(t, y)$ is a function of $r^2/t$ by the usual parabolic scaling. For brevity, we denote this particular case by  $\Theta_n(r^2/t)$, and we include a quick discussion about this in Subsection \ref{subsec:Theta} below. 

\subsection{A few estimates on concentration near the zero set}

We begin by studying the question of spatial eigenfunction distribution in terms of certain $L^p$-norms over neighbourhoods around the zero set of $ \varphi_\lambda $. 
Our first result is the following:

\begin{theorem}\label{thm:conc_L1}
	Let $M$ be a smooth closed Riemannian manifold with sectional curvature bound $K_1 \leq \text{Sec} \leq K_2$. There exists a  positive constant $ C_1 $, 
	depending only on $ K_i$, such that for every small enough positive numbers $ t, r $ (independent of $\lambda$) satisfying $0 < t \leq r^2  $, we have
	
	\begin{equation}\label{ineq:L1_conc}
		\| \varphi_\lambda\|_{L^1(T_r)} \geq \left(1 - e^{- t\lambda} - C_1\Theta_n(r^2/t) \right)\| \varphi_\lambda\|_{L^1(M)}.
	\end{equation}
\end{theorem}

Letting $|K_i| \to 0$, and making $r$ small, one can guarantee that $C_1$ is arbitrarily close to $ 1 $, which agrees with the Euclidean case; 
 this follows  from near diagonal small time Gaussian heat kernel bounds and volume comparison of balls, which depends on curvature bounds.
We verify Theorem \ref{thm:conc_L1} by introducing a certain diffusion process on each nodal domain of $\varphi_\lambda$ and estimating its solution in an appropriate way. 
The comparability constant $C_1$ in 
(\ref{ineq:L1_conc}) 
is related to the constants 
in the short-time upper Gaussian heat kernel bounds for $M$ (see (\ref{ineq:heat_upper}) and 
(\ref{ineq:heat_kernel}) below), $r$ and the dimension $n$. 

 To get a nicer looking expression out of (\ref{ineq:L1_conc}), 
 we set $r = r_0\lambda^{-1/2}, t = t_0\lambda^{-1}$, and $r_0^2/t_0 = c$, a large constant, such that we have $\Theta_n\left(r_0^2/t_0\right) \leq \frac{1}{2C_1}\left(1 - e^{-t_0}\right)$, whence 
 it is clear that 
 the right hand side in (\ref{ineq:L1_conc}) is $\geq \frac{1}{2}(1 - e^{-t_0})$. In addition, if the constant $t_0$ is chosen sufficiently smaller than $1$, this implies that 
$$
\| \varphi_\lambda \|_{L^1(T_r)} \geq \frac{1}{4}t_0\| \varphi_\lambda\|_{L^1(M)} =  \frac{1}{4c} r_0^2\| \varphi_\lambda\|_{L^1(M)}.
$$

Next, we observe that in dimension $n = 2$, one can use heat equation techniques in conjunction with harmonic measure theory (the latter not being
available in higher dimensions) to further refine (\ref{ineq:L1_conc}) and obtain the reverse estimate:
\begin{theorem}\label{thm:conc_L1_n=2} Let $M$ be a smooth closed Riemannian surface.
Given a positive constant $C_2$, one can find  positive constants $C_3, \lambda_0$ such that for $\lambda \geq \lambda_0$, 
we have
	\begin{equation}\label{eq:L1_conc}
	\| \varphi_\lambda\|_{L^p\left(T_{r_0\lambda^{-1/2}}\right)} \leq C_3 \left(1 - e^{-pt_0}\right)^{1/p}\| \varphi_\lambda\|_{L^p(M)}, \text{  for } p \in [1, \infty),
	\end{equation}
where $t_0 \leq r_0^2 \leq C_2$.
\end{theorem} 
The comparability constant $ C_3$ in  (\ref{eq:L1_conc}) depends on the ratio $r_0^2/t_0$, and on the distortions of 
wavelength balls in $M$ to 
the unit disk in $\CC$ under quasi-conformal mappings, which in turn is dependent on the geometry of $(M, g)$.

Observe that Theorems \ref{thm:conc_L1} and \ref{thm:conc_L1_n=2} together give a two-sided estimate for $p = 1$. 
A few comments are in place.

\begin{itemize}
    \item It is important to discuss the scale of $r$ in $T_r$ in Theorems \ref{thm:conc_L1} and \ref{thm:conc_L1_n=2}. As is well-known, the nodal set of any eigenfunction $\varphi_\lambda$ is $C(M, g) \lambda^{-1/2}$-dense in $M$ for some $C(M, g)$ independent of $\lambda$\footnote{There are two popular proofs of this fact, one uses domain monotonicity of Dirichlet eigenvalues, and the other uses Harnack inequality on the harmonic function $e^{\sqrt{\lambda}t}\varphi_\lambda(x)$ in $\RR \times M$.}. Clearly, in the regime $r \geq C(M, g)\lambda^{-1/2}$, we have that $T_r = M$. So in particular, the bound in (\ref{ineq:L1_conc}) is only non-trivial when $r << C(M, g)\lambda^{-1/2}$.
	\item Theorem \ref{thm:conc_L1} and Theorem \ref{thm:conc_L1_n=2} can be viewed as a rough ``aggregated'' doubling/growth condition. We recall that doubling indices of the type
	\begin{equation}
		\log \frac{\sup_{B(x, 2r)}|\varphi_\lambda|}{\sup_{B(x, r)}|\varphi_\lambda|}, \quad x \in M, \quad r > 0,
	\end{equation}
		have found extensive applications in the study of vanishing orders and nodal volumes (cf. \cite{DF}, \cite{L1}, \cite{L2}, \cite{S}). A result of Donnelly-Fefferman (\cite{DF}) states that such doubling indices are at most at the order of $ \sqrt{\lambda} $. It is expected, however, that such a saturation happens rarely (cf. \cite{DF}, \cite{L1}). In fact, on average the doubling indices should be bounded by a uniform constant (independent of $ \lambda $) - we refer to the works of Nazarov-Polterovich-Sodin and Roy-Fortin (\cite{NPS}, \cite{R-F}).
	
	\item 
	In the generic case the eigenfunction $ \varphi_\lambda $ is a Morse function and, in particular, one expects that the nodal set has a relatively small singular set (cf. \cite{S}). In fact, one would guess that the gradient $ \nabla \varphi_\lambda $ is sufficiently large on a big subset of $ N_{\varphi_\lambda} $ - an easy example is given by $ \sin(\sqrt{\lambda} x) $ on $ \mathbb{S}^1 $, where the gradient is $ \sim \sqrt{\lambda} $ at the nodal set.
	
	In the case of a Riemannian surface the above Theorems \ref{thm:conc_L1} and \ref{thm:conc_L1_n=2} suggest that if 
	we choose $t = t_0/\lambda, r = r_0/\sqrt{\lambda}$, and $r_0^2$ large compared to $t_0$, where $t_0, r_0$ are constants, then
	\begin{equation}
	\| \varphi_\lambda \|_{L^1(T_{r })} \sim r_0^2 \| \varphi_\lambda \|_{L^1(M)}.
	\end{equation}
	
	Having in mind estimates on nodal volumes (cf. \cite{S}, \cite{L1}, \cite{L2}), one would guess that the left-hand side approximately grows as $ r \sqrt{\lambda} $ times the average of $ |\nabla \varphi_\lambda| $ on $ N_{\varphi_\lambda} $. Thus, one heuristically recovers the following formula due to Sogge-Zelditch (cf. \cite{SZ1}):
	
	\begin{proposition} \label{prop:SZ}
		For any function $ f \in C^2(M) $, one has
		\begin{equation}
		\int_M \left( (\Delta + \lambda) f \right) |\varphi_\lambda| dx = 2 \int_{N_{\varphi_\lambda}} f |\nabla \varphi_\lambda| d\sigma.
		\end{equation}
	\end{proposition}
	
	On the other hand, it seems likely that the proofs of Theorems \ref{thm:conc_L1} and \ref{thm:conc_L1_n=2} can also utilize Proposition \ref{prop:SZ} with a suitable choice of $ f $, namely, a solution to an appropriate heat/diffusion equation. 
	
	\item The bound in Theorem \ref{thm:conc_L1} might look slightly counter-intuitive at first glance. Indeed, mass should not be expected to concentrate around nodal sets in particular (where the mass is smallest) but one could presumably have some mass concentration if $|\nabla\varphi_\lambda|$ is large on the nodal set. As an illustrative example, consider the case of the highest weight spherical harmonics $(x_1 + ix_2)^l$ on $S^2$ (as is well-known, the corresponding eigenvalue is $l(l + 1)$).  One can calculate that (up to constants) $\displaystyle{\| (x_1 + ix_2)^l \|_{L^p(S^2)}^p = \frac{\Gamma\left( \frac{lp}{2} + 1\right)}{\Gamma\left( \frac{lp}{2} + \frac{3}{2}\right)} \sim \left(\frac{lp}{2} + 1\right)^{-\frac{1}{2}}}$ (see \cite{Z1}, Chapter 4) for high enough $l$. However, it follows from another standard  computation that after converting to spherical coordinates, for high enough $l$, we have that
\begin{align*}
    \| \Re (x_1 + ix_2)^l \|_{L^p(T_{1/l})}^p &  = \int_{T_{1/l}} |\sin^{lp}\theta| |\sin^pl\omega| \;d\theta d\omega  \geq c_1(p)\int_{\theta = \pi/2 + (lp)^{-\frac{1}{2}}}^{\pi/2 - (lp)^{-\frac{1}{2}}} |\sin^{lp}\theta|\; d\theta \\
    & = c_1(p)\int_{-(lp)^{-\frac{1}{2}}}^{ (lp)^{-\frac{1}{2}}} |\cos^{lp}\theta|\; d\theta \gtrsim_p c_1(p)(lp)^{-\frac{1}{2}},
\end{align*} 
where we use Taylor series expansion for estimating the powers of the sine function from below, and the constant $c_1(p) \searrow 0$ as $p \nearrow \infty$.  
 A similar calculation shows that for large enough $l$,
 \begin{align*}
    \| \Re (x_1 + ix_2)^l \|_{L^p(T_{1/l})}^p & \lesssim c_2(p)\int_{0}^{\frac{\pi}{2}} |\cos^{lp}\theta|\; d\theta \sim c_2(p) \frac{\Gamma\left( \frac{lp}{2}\right)}{\Gamma\left( \frac{lp}{2} + \frac{1}{2}\right)} \sim_p  (lp)^{-1/2}, 
\end{align*}
where the constant $c_2(p) \searrow 0$ as $p \nearrow \infty$. This shows  mass (non)concentration properties around nodal sets for Gaussian beams on $S^2$. 
	
	\item Observe that for $p = 1$, Theorem \ref{thm:conc_L1_n=2} effectively says that the measures $\varphi^{+}\;dx$ and $\varphi^-\; dx$ defined by the positive and negative parts respectively of the eigenfunction $\varphi_\lambda$ have sufficient mass wavelength distance away from each other. In other words, the work done to ``move'' the positive mass to the negative mass should be large. This has been subsequently used in \cite{M} to prove a conjectured sharp lower bound on the Wasserstein distance between the measures (see \cite{St1}).

\item Lastly, one can ask the natural question as to whether versions of Theorem \ref{thm:conc_L1} could hold for $p > 1$. It is immediately clear that such a result cannot be true for $p \nearrow \infty$ (unless the constant in the estimate is going to zero). It seems intriguing to speculate whether there is a regime of low enough $p$ for which there is sufficient $L^p$-mass concentration around the nodal set. The idea is that for low enough $p$, it is still possible have a ``highest weight'' scenario, as opposed to the ``zonal'' scenario for high values of $p$, where only one point which is 
 located ``deep inside'' the nodal domain away from the nodal set influences the whole picture. But we are not sure at this point whether this happens and what this threshold of $p$ is, and it would be a really interesting question to determine the exact regime and compare it with \cite{S2}. 
\end{itemize}

\subsection{Avoided crossings phenomena in more general situations}

Now, we discuss the {\em avoided crossings} of nodal sets of eigenfunctions. Roughly, this refers to the phenomenon 
that nodal domains $\Omega_\lambda$ cannot be too ``narrow'' over a long distance, which is aimed at somehow quantifying how 
``thin'' nodal domains can be. This phenomenon also seems to be of physical interest; for an introduction, 
see \cite{MSG} and references therein, as well as 
Section $5$ of \cite{St}.

It is known that two nodal lines on a closed surface cannot approach each other much more 
than a distance of $1/\lambda$ along a line of length $\gtrsim 1/\sqrt{\lambda}$ (for a precise statement, see Theorem 
\ref{thm:Man} below, which is proved in \cite{Ma1}). 
Our main result in this direction is Theorem \ref{thm:thin_nod} giving an extension of the results of Mangoubi
and Steinerberger in dimension $n = 2$. 

The central motivating question is the following: can a nodal domain locally consist of arbitrarily thin pieces or branches?

\begin{definition}
	Let $ \Omega $ be a nodal domain and let $ B_R(x) $ be a ball of radius $ R $ with center $ x \in \Omega $. We define the local nodal domain as
	\begin{equation}
		\locnod := \Omega \cap B_R(x).
	\end{equation}
\end{definition}
By way of notation, let  $\Omega^s_{\text{loc}} := \Omega \cap B_{sR}(x)$, where $0 < s < 1$. Without loss of generality we are going to assume that $\locnod$ and $\locnod^s$ are connected, otherwise our analysis will work separately for each connected component of $\locnod^s$. We assume that $R \leq R_0$, 
 which is a positive constant dependent on the geometry of $(M, g)$. Now, we bring in the following 

\vspace{-2.5mm}
    \begin{equation}\label{cond:thin}
        \begin{minipage}{0.8\textwidth}
            \emph{Local thinness condition:} suppose one can cover the local nodal domain $\locnod$ with cubes $Q_j$ of side length $r \sim \lambda^{-\alpha}, \alpha > \frac{1}{2}$, such that the interiors of no two cubes intersect, and $\pa Q_j \setminus \locnod$ has positive length for all $j$.  
        \end{minipage} \tag{LTC}
    \end{equation}
    \vspace{0.5mm}
Then, we have the following result: 
\begin{theorem}\label{thm:thin_nod}
	Let $ \locnod $ be a local nodal domain inside a smooth closed Riemannian surface $M$ satisfying (\ref{cond:thin}) above. There exists a positive constant $C$  depending on $s$ and the geometry of the surface $M$, such that
	\begin{equation}\label{ineq:loc_diam_est}
		\diam(\locnod^s ) \leq C \lambda^{\frac{1}{2} - \alpha} \log \lambda,
	\end{equation}
	where $\diam(\locnod^s)$ refers to the diameter of $\locnod^s$.
\end{theorem}

As an example, the above theorem applies when the local nodal domain $\locnod$ is contained in the $r$-tubular neighbourhood of an embedded tree. One speculates that another class of examples should be given by domains with a large asymmetry constant (see Definition \ref{def:AA} below for the notion of asymmetry) and a small inner radius $\inrad(\locnod)$ of the local nodal domain. Intuitively it seems that because of the asymmetry only an upper bound on the inner radius should be enough to force local thinness. 
Roughly, Theorem \ref{thm:thin_nod} should be seen 
as a two-dimensional analogue of the results in Section 5 of \cite{St}. The latter apply to long thin nodal domains which are ``essentially one-dimensional''. 
We also observe that Theorem \ref{thm:thin_nod} refines the  result of Mangoubi (\cite{Ma1}, Theorem \ref{thm:Man} below). Heuristically, $\locnod$ can consist of a (possibly disjoint) union of a large number of thin branches or ``veins'' which proliferate 
throughout $B_R$, and in particular, have a relatively large volume compared to $|B_R|$ (see the diagram below). This does not seem to violate 
Mangoubi's result (Theorem \ref{thm:Man} below), but is ruled out by Theorem \ref{thm:thin_nod}. As a particular example, consider a ball $B$ of wavelength radius, and a domain $\Omega$ that 
consists of (up to a small constant) $\lambda^{-\frac{1}{2} + \alpha}$ many equispaced vertical and horizontal  ``strips'', where the width of each strip is $\sim\lambda^{-\alpha}$. Clearly, $|\Omega\cap B| \sim \lambda^{-1}$, and $\frac{|\Omega \cap B|}{|B|} \sim 1$. If we ask the question if $\Omega \cap B$ can be a local nodal domain, we are unable to conclude it from relative volume considerations given in \cite{Ma1}. However, if $\Omega\cap B$ were to be a local nodal domain, then we would have $\diam (\Omega \cap B) \sim \lambda^{-\frac{1}{2}}$, and Theorem \ref{thm:thin_nod} dictates that this cannot be possible, as (upto logarithmic factors) $\diam (\Omega \cap B) \lesssim \lambda^{\frac{1}{2} - \alpha}$. 

Recently, the question of relative volume estimates for local nodal domains has been solved in all dimensions in \cite{CLMM}, extending and superseding the result of \cite{Ma1}. Theorem 1.2 of \cite{CLMM} gives a lower bound on the volume of each local nodal domain, and proves that the number of connected components of $B_r \setminus N_\varphi$ is $\lesssim \lambda^{\frac{n}{2}}r^n + \lambda^{\frac{n - 1}{2}}$ (we note that this is the optimal result in this direction which completely closes the line of investigation started by Donnelly-Fefferman). However, as pointed above, one can keep the volumes and number of connected components of local nodal domains constant, but make the nodal domains thinner and more dense. As the calculation above demonstrates, such a phenomenon might not be detected solely via volume based tests. So, our result gives a somewhat different point of view to the problem of avoided crossings. 

Unfortunately, the proof of Theorem \ref{thm:thin_nod} is essentially two dimensional in nature, depending crucially on the following fact: if the starting point of a Brownian particle is close to an obstacle, then the probability of the Brownian particle hitting the obstacle is high. As is well-known, such a statement cannot hold in higher dimensions because one can construct spikes with arbitrarily low capacity. In other words, in dimensions $\geq 3$, a Brownian particle can be very close to the nodal set, but still have very low probability of hitting it.  
\begin{center}
	\includegraphics[width=80mm,scale=0.8]{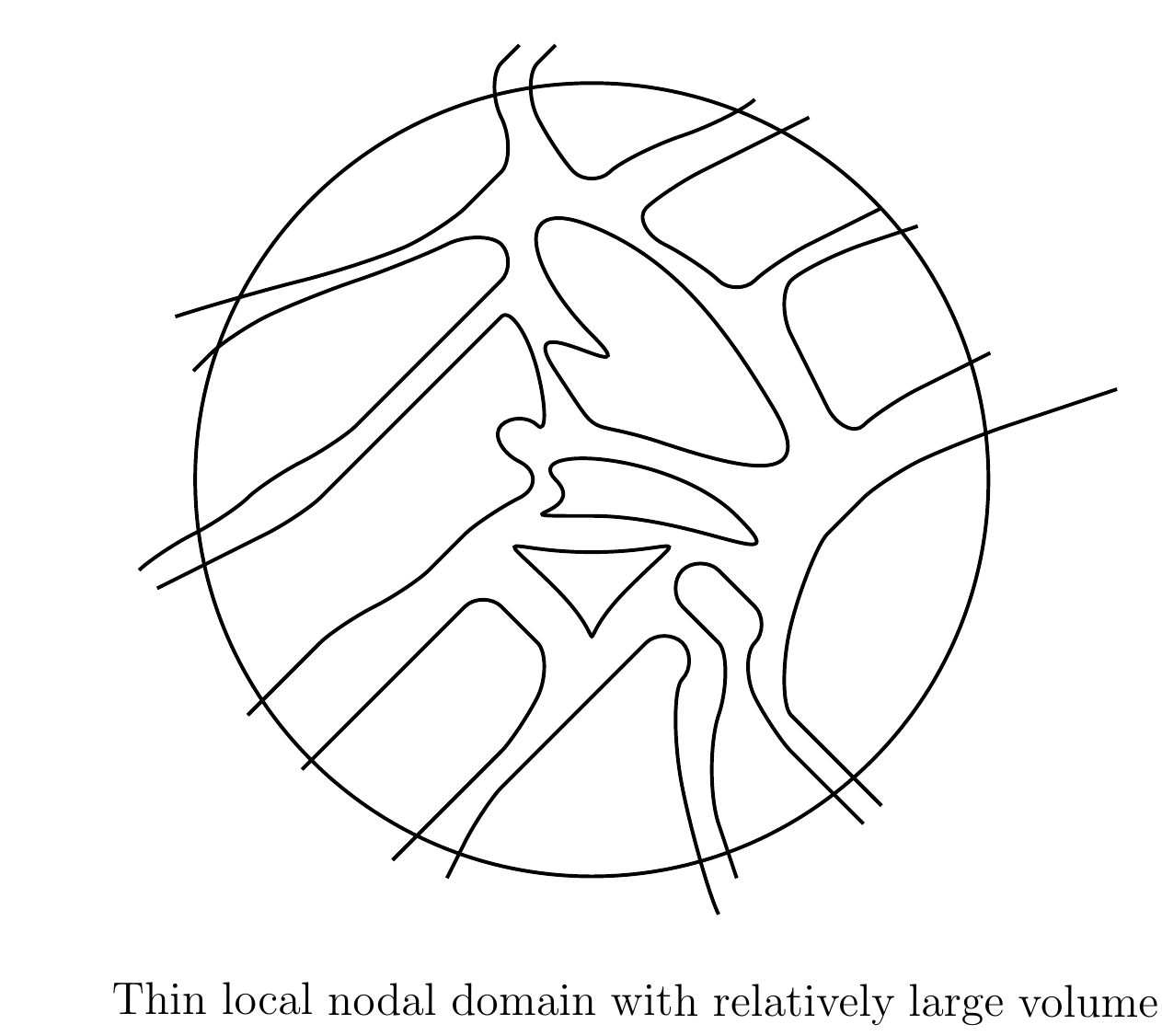}\label{diag2}
\end{center}

\vspace{7mm}


\subsection{Decay of Laplace eigenfunctions in domains with narrow branches}

Now, we take up the question of tunneling of {\em low energy} Laplace eigenfunctions through narrow regions
in a domain. 
There is a substantial amount of recent literature on different 
variants and interpretations of the above question. For a broad overview of some of the questions and ideas involved, we refer the reader to Section $7$ of \cite{GN} and references therein. We mention in particular the recent articles \cite{NGD}, \cite{DNG} and the slightly less recent \cite{vdBB} and \cite{BD}. Inherent in all the above mentioned references seems to be the following heuristic (physical) fact: there is a sharp decay of an eigenfunction in a ``narrow tunnel-shaped'' (or horn-shaped) region when the wavelength of the eigenfunction is larger than the width of the tunnel. In other words, eigenfunctions cannot tunnel effectively through subwavelength openings. This leads to the localization of low frequency eigenfunctions in relatively thicker regions of a domain.

We have the following preliminary estimate:
\begin{theorem}\label{thm:decay_narrow_branch}
Let $\Omega \subseteq \RR^n$ be a bounded domain, $x \in \Omega$, and $\varphi$ be a Dirichlet eigenfunction of $\Omega$ corresponding to the eigenvalue $\lambda$. 
 Then, we have that 
\beq\label{ineq:thin_dom_sup_bound}
\varphi(x) \leq  e^{\frac{\lambda}{\lambda_0}}\left(1 -  
\psi_{\RR^n \setminus \Omega}\left(\lambda_0^{-1}, x\right)\right)\| \varphi\|_{L^\infty(\Omega)}, \text{  where } \lambda_0 > 0. 
\eeq
\end{theorem}

We would like to use (\ref{ineq:thin_dom_sup_bound}) in the special setting that $\lambda \ll \lambda_0$, and $x$ is inside a narrow bottleneck region of $\Omega$. In this case, an interesting feature of the above result is that we effectively do not need any assumptions on the shape of the domain $\Omega$ away from $x$ beyond length scales $\sim \lambda_0^{-1/2}$. As an example, even a narrow localized bottleneck in an otherwise bulky domain is enough to force local decay. 
Below we discuss a few natural notions of quantifying the aforementioned ``narrowness'' to indicate the scope of applicability of Theorem \ref{thm:decay_narrow_branch}.
\begin{itemize}
    \item 

Let $\frac{1}{\sqrt{\mu}}$ be the radius of the largest ball around $x$ that is fully contained in $\Omega$. 
We are looking for a natural way to quantify the fact that $x$ is contained in a ``narrow tentacle'' of the domain $\Omega$ that can be visualized as an ``octopus'', the eigenvalue $\lambda$ 
is less than or equal to $\lambda_0$, and $\mu$ is large compared to $\lambda_0$. 
Now, suppose we are given an asymmetry constant $\alpha$ of the domain, 
whose definition we recall here for the sake of completeness: \begin{definition}\label{def:AA}
	A bounded domain $\Omega \subseteq \RR^n$ is said to satisfy the asymmetry assumption with coefficient $\alpha$ (or $ \Omega $ is $ \alpha $-asymmetric) if for all $x \in \partial \Omega$, and all $r_0 > 0$, 
	\begin{equation}
	\frac{|B_{r_0}(x) \setminus \Omega|}{|B_{r_0}(x)|} \geq \alpha.
	\end{equation}
\end{definition}
This condition seems to have been introduced in \cite{H}, see also \cite{Ma3, GM3} for applications of this concept. 
There are also variants of Definition \ref{def:AA} in terms of capacity, which achieve a similar effect, see \cite{vdBB, MS}. From our perspective, the notion of asymmetry is useful as it basically rules out pathologies involving narrow ``spikes'' (i.e. needle-like sets with relatively small volume) entering deeply into $\Omega$, allowing us to force local narrowness with control on the inner radius alone. Now we might posit that we are looking at 
those points $x \in \Omega$ for which there exists a number $\tilde{r}$ (possibly  $\tilde{r} := \tilde{r}\left(\alpha 
\right)$) such that 
\begin{equation}
  \psi_{\RR^n \setminus B\left(x, \frac{\tilde{r}}{\sqrt{\mu}}\right)} (t, x)
\leq \psi_{\RR^n \setminus \Omega} (t, x) \text{ for time scales } t \lesssim \lambda_0^{-1}.  
\end{equation}
In this situation, $\Theta_n\left(\frac{\tilde{r}^2\lambda_0}{\mu }\right)$ is close to $1$ (since $\mu$ has been assumed large compared to $\lambda_0$), and (\ref{ineq:thin_dom_sup_bound})
reads:
\begin{equation*}
   \varphi(x) \leq  e^{\frac{\lambda}{\lambda_0}}\left(1 -  
\Theta_n\left(\frac{\tilde{r}^2\lambda_0}{\mu }\right)\right)\| \varphi\|_{L^\infty(\Omega)},
\end{equation*}
which gives a non-trivial conclusion. 
A 
closely related alternative could be to replace balls by cubes in the above construction, and stipulate that there exists a $\tilde{r}$ (possibly depending on $\alpha$), and a cube $C_{\tilde{r}}$ of side length $\frac{\tilde{r}}{\sqrt{\mu}}$ centered at $x$ such that $\psi_{\Omega}(t, x) \leq \psi_{C_{\tilde{r}}}(t, x)$ for $t \lesssim  \lambda_0^{-1}$. In that case, we would have similar conclusions 
with $\Theta_n$ 
replaced by $\left( 4 \Phi\left(-\frac{\tilde{r}}{2\sqrt{\mu/\lambda_0}} \right) \right)^n$ (see Subsection \ref{subsec:Theta} below).

We make a few parenthetical remarks about asymmetry.  Observe that convex domains trivially satisfy the asymmetry condition with coefficient $\tau = \frac{1}{2}$. Also, note that in general, a domain does not need to be asymmetric with any positive coefficient. However, a deep result of Mangoubi (see \cite{Ma3}) indicates that nodal domains are asymmetric with $\tau = \frac{C}{\lambda \left(\log \lambda\right)^{1/2}}$ for Riemannian surfaces and  $ \tau = \frac{C}{\lambda^{(n-1)/2}} $  for dimensions $n \geq 3$.  This gives one of the few semi-local characteristics or tests (known to us) that can distinguish a nodal domain from any general domain. As an particular example, it implies that a subdomain of $\Omega$ with a slit cannot be a nodal domain for any eigenfunction of $\Omega$. The current authors have also used $\tau$-asymmetry in \cite{GM3} to study an obstacle problem. There is also a variant of Definition \ref{def:AA} in terms of capacity, which achieves a similar effect (see \cite{vdBB} and \cite{MS}).
\item Another natural notion of narrowness is to insist that $B\left(x, \lambda_0^{-1/2}\right) \cap \Omega $ is contained in a  $\frac{1}{\sqrt{\mu}}$-tubular neighbourhood of a line segment of length $\geq c_1\lambda_0^{-1/2}$ and  curvature controlled by $\kappa$. This is a generalization of ``horn-shaped'' domains considered in \cite{DS, BD}, and effectively captures the idea that $x$ is contained in a bottleneck of the domain. Now it is clear that there is a positive constant $c := c(\kappa, c_1)$ such that 
\begin{equation*}
    \psi_{\RR^n \setminus \Omega}\left(\lambda_0^{-1}, x\right) \geq  c\;\Theta_{n - 1}\left( \frac{\lambda_0}{\mu} \right),
\end{equation*}
whence (\ref{ineq:thin_dom_sup_bound})
reads:
\begin{equation*}
   \varphi(x) \leq  e^{\frac{\lambda}{\lambda_0}}\left(1 -  
c\;\Theta_{n - 1}\left(\frac{\lambda_0}{\mu }\right)\right)\| \varphi\|_{L^\infty(\Omega)}.
\end{equation*}
For example, if 
$c_1$ is large enough, then $c$ is arbitrarily close to $1$.
\end{itemize}



\subsection{A few comments on $\Theta\left( r^2/t\right)$}\label{subsec:Theta}
The expression $\Theta_n\left(r^2/t\right)$ will appear in several estimates in our paper. It seems that there is no nice closed formula for $\Theta_n$ in the literature. We first state the following estimate, found in \cite{Ke}
\beq\label{JKF}
\Theta_n\left(r^2/t\right) = 1 - \frac{1}{2^{\nu - 1}\Gamma (\nu + 1)}\sum^\infty_{k = 1}\frac{j_{\nu, k}^{\nu - 1}}{J_{\nu + 1}(j_{\nu, k})}e^{-\frac{j^2_{\nu, k}\;t}{2r^2}}, \quad \nu > -1,
\eeq
where $\nu = \frac{n - 2}{2}$ is the ``order'' of the Bessel process, $J_\nu$ is the Bessel function of the first kind of order $\nu$, and $0 < j_{\nu, 1} < j_{\nu, 2} < .....$ is the sequence of positive zeros of $J_\nu$.

Another expression that can be quickly calculated from the definition via heat kernel is the following:
\beq
\Theta_n\left(r^2/t\right) = 
\frac{1}{\Gamma\left(\frac{n}{2}\right)}\int_0^t \Gamma\left(\frac{n}{2}, \frac{r^2}{4s}\right) \; ds,
\eeq
where $\Gamma(s, x)$ is the upper incomplete Gamma function defined by 
\beq
\Gamma (s, x) = \int^\infty_x e^{-t}t^{s - 1} dt.
\eeq
Our main regime of interest is $t \leq r^2$, and $r$ being small. Via asymptotics involving the incomplete Gamma function, given $\varepsilon > 0$, one can choose $r^2/t = c$ large enough to get the following crude upper bound:
\beq\label{ineq:asymp_gamma}
\Theta_n\left(r^2/t\right) \leq \frac{1 + \varepsilon}{\Gamma\left(\frac{n}{2}\right)} c^{\frac{n}{2} - 1} \; t\; e^{-\frac{r^2}{4t}}.
\eeq
Lastly, another crude upper bound for $\Theta_n$ can be obtained in the following way. Inscribe a hypercube inside the sphere of radius $r$, whence each side of the hypercube has length $\frac{2r}{\sqrt{n}}$. So by the reflection principle, the probability of a Brownian particle escaping the cube is given by 
\beq\label{eq:ref_prin}
\left(4\Phi\left(-\frac{r}{\sqrt{nt}}\right)\right)^n, 
\eeq
where 
\beq
\Phi(s) = \int_{-\infty}^s \frac{1}{\sqrt{2\pi }}e^{-\frac{s^2}{2}}\; ds
\eeq
is the cumulative distribution function of the one dimensional normal distribution. On calculation, it can be checked that in the regime $r^2/t \geq n$, the expression in (\ref{eq:ref_prin}) translates to 
\beq
\Theta_n\left(r^2/t\right) \leq \frac{2^{3n/2}}{\pi^{n/2}} e^{-\frac{r^2}{2t}}.
\eeq

\section{Heat content} \label{sec:tools}

We first recall (see \cite{HS}) that the nodal set can be expressed as a union
\beq
N_\varphi =  H_\varphi \cup \Sigma_\varphi,
\eeq
where $H_\varphi$ is a smooth hypersurface and  
$$
\Sigma_\varphi := \{ x \in N_{\varphi_\lambda} : \nabla \varphi_\lambda(x) = 0 \}
$$
is the singular set which is countably $(n - 2)$-rectifiable. Particularly, in dimension $n = 2$, the singular set $\Sigma_\varphi$ 
consists of isolated points.

Note that we can express $M$ as the disjoint union 
$$ 
M = \bigcup_{j = 1}^{j_0} \Omega_j^+ \cup \bigcup_{k = 1}^{k_0} \Omega_k^- \cup N_{\varphi_\lambda},
$$
where the $\Omega_j^+$ and $\Omega_j^-$ are the positive and negative nodal domains respectively of $\varphi_\lambda$. 

Given a  domain $\Omega \subset M$, consider the solution $p_t(x)$ to the following diffusion process:
\begin{align*}
(\pa_t - \Delta)p_t(x) & = 0, \text{    } x \in \Omega\\
p_t(x) & = 1, \text{    } x \in \pa \Omega\\
p_0(x) & = 0, \text{    } x \in \Omega.
\end{align*}
By the Feynman-Kac formula, this diffusion process can be understood as the probability that a Brownian motion 
particle started in $x$ will hit the boundary within time $t$. The quantity 
$$
\int_\Omega p_t(x) dx
$$ 
is called the {\em heat content} of $\Omega$ at time $t$. It can be thought of as a soft measure of the ``size'' of the 
boundary $\pa \Omega$. 

Now, fix an eigenfunction $\varphi_\lambda$ (corresponding to the eigenvalue $\lambda$) and a nodal domain $\Omega$, so that 
$\varphi_\lambda > 0$ on $\Omega$ without loss of generality. Calling $\Delta_\Omega$ the Dirichlet Laplacian on $\Omega$ and 
setting $\Phi (t, x) := e^{t\Delta_\Omega}\varphi_\lambda (x)$, we see that $\Phi$ solves 
\begin{align}\label{eq:Heat_Flow}
(\pa_t - \Delta_\Omega)\Phi (t, x) & = 0, x \in \Omega\nonumber\\
\Phi (t, x) & = 0, \text{  on   } \{\varphi = 0\}\\
\Phi (0, x) & = \varphi_\lambda (x), \text{    } x \in \Omega.\nonumber
\end{align}
By the Feynman-Kac formula, this deterministic diffusion process can be understood as  an expectation over the behaviour of a 
random variable; for more details see Theorem 2.1 of \cite{GM1} which has a precise formulation of the Feynman-Kac formula for 
the compact setting and the proper boundary regularity. Using the Feynman-Kac formula we have,
\beq\label{eq:F-K}
e^{t\Delta_\Omega}f(x) = \mathbb{E}_x(f(\omega(t))\phi_{\Omega}(\omega, t)), t > 0, 
\eeq
where $\omega (t)$ denotes an element of the probability space of Brownian motions starting at $x$, $\mathbb{E}_x$ is the 
expectation with regards to the (Wiener) measure on that probability space, and 
$$
\phi_\Omega(\omega, t) = 
\begin{cases}
1, & \text{if } \omega ([0, t]) \subset \Omega\\
0, & \text{otherwise. }
\end{cases}
$$
In particular, $p_t(x) = 1 - \mathbb{E}_x(\phi_\Omega(\omega, t))$.

\section{Concentration around the nodal set: proofs of Theorems \ref{thm:conc_L1} and \ref{thm:conc_L1_n=2}}\label{sec:mass_conc}
We first state and prove a preliminary lemma.
\begin{lemma} \label{lem:SZ-for-pt(x)}
	The following formula holds
	\begin{equation}\label{eq:F_K}
		\int_\Omega p_t(x) \varphi_\lambda (x) dx = (1 - e^{-t\lambda}) \int_\Omega \varphi_\lambda (x) dx.
	\end{equation}
	As a consequence,
	\begin{equation}
		\|p_t \varphi_\lambda \|_{L^1(M)} = (1 - e^{-t \lambda}) \| \varphi_\lambda \|_{L^1(M)}.
	\end{equation}
\end{lemma}

\begin{proof}
	Writing $K_\Omega(t, x, y)$ as the heat kernel for $\Delta_\Omega$ and observing that 
	\beq\label{eq:pt_f}
	\mathbb{E}_x(\phi_\Omega(\omega, t))  = \int_\Omega K_\Omega(t, x, y) dy = 1 - p_t(x),
	\eeq
	we have that on a nodal domain $\Omega$ (see also Section 3.2 of \cite{St}), 
	\begin{equation}
	\int_\Omega p_t(x) \varphi_\lambda (x) dx = \int_\Omega (1 - e^{t\Delta_\Omega})\varphi_\lambda (x) dx = (1 - e^{-t\lambda}) 
	\int_\Omega \varphi_\lambda (x) dx.
	\end{equation}

	\noindent Adding over all nodal domains and using (\ref{eq:F_K}) we get, 
	\begin{align*}
	\| p_t\varphi_\lambda\|_{L^1(M)} & = \int_M |p_t(x) \varphi_\lambda(x)| dx \\
	& = \sum_{j = 1}^{N^+_\lambda} \int_{D^+_j} p_t(x)\varphi_\lambda(x) dx - \sum_{j = 1}^{N^-_\lambda} \int_{D^-_j} 
	p_t(x)\varphi_\lambda(x) dx\\
	& = (1 - e^{-t\lambda})\| \varphi_\lambda\|_{L^1(M)}.
\end{align*}

\end{proof}

\begin{remark}
	Lemma \ref{lem:SZ-for-pt(x)} can also be derived from Proposition \ref{prop:SZ} by plugging in $ f = p_t(x) $ and then using the definition of $ p_t(x) $.
\end{remark}
\subsection{Proof of Theorem \ref{thm:conc_L1}}
Our goal is to properly estimate the integral $\| p_t\varphi_\lambda\|_{L^1(M)}$. Recall that Varadhan's 
large deviation formula states that 
\beq\label{eq:Var_lar_dev}
\lim_{t \to 0} -4t\log K(t, x, y) = \text{dist}(x, y)^2, 
\eeq
where $K(., .,.)$ is the heat kernel on the manifold $M$ (for more details, see \cite{V1}). This can be interpreted as the heuristic fact that at short time scales a typical 
Brownian particle travels a distance
$\sim \sqrt{t}$ in time $t$. So, if we look at time scales $t \sim \lambda^{-1}$ and distance scales $r \sim \lambda^{-1/2}$ 
from the nodal set, we should expect the integral $\| p_t \varphi_\lambda\|_{L^1(M)}$ to effectively ``localize'' in a 
$r$-tubular neighbourhood of $N_\varphi$. Below, we make this precise. Write 
\beq 
\int_\Omega p_t(x)\varphi_\lambda(x) dx = \int_{T_{r
}} p_t(x)\varphi_\lambda(x) dx + 
\int_{\Omega \setminus T_{r
}} p_t(x)\varphi_\lambda(x) dx.
\eeq

We readily have that,  
\beq
\int_{T_{r
}} |p_t(x)\varphi_\lambda(x)| dx \leq \int_{T_{r
}} |\varphi_\lambda (x)| dx.
\eeq

What remains is to estimate the integral $\int_{M \setminus T_{r
}} |p_t(x)\varphi_\lambda(x)| dx$. Deep inside 
any 
nodal domain $\Omega$, we expect that $ p_t(x) $ is very small. In fact, we have the following upper bound:
\begin{lemma}
    If the manifold $M$ has sectional curvature bounds $K_1 \leq \text{Sec} \leq K_2$, there exists a universal positive constant 
    $c_1(K_1, K_2)$, 
	such that
	\begin{equation}\label{ineq:heat_upper}
		|p_t(x) | \leq c_1\Theta_n\left(\frac{\text{dist} (x, \pa \Omega)^2}{t}\right). 
	\end{equation}
\end{lemma}

\begin{proof}
	According to Section \ref{sec:tools}, $ p_t(x) $ is a solution to a certain heat/diffusion process, and thus can be interpreted as a Brownian motion hitting probability - namely, the probability that a Brownian motion started at $ x $ hits the boundary $ \partial \Omega $ within time $ t $. In order that such an event occurs, the Brownian particle has to escape the sphere 
	$B(x, \text{dist} (x, \pa \Omega))$ first, and the latter is given by $ \Theta_n\left(\frac{\text{dist} (x, \pa \Omega)^2}{t}\right)$, supposing for the moment that our space is locally Euclidean. 
	
	Since we are in 
 	the curved setting, one needs to use in conjunction the comparability of hitting probabilities on Euclidean spaces and that 
	on compact manifolds. For such a comparability result which works in the regime of small distance scales $r$ and time scales $t \leq O(r^2)$, and which uses the concept of Martin capacities, see Theorem $ 3.1 $ of \cite{GM1}. The dependence of the constant $c_1$ on the geometric data is also clear from the proof of Theorem $3.1$ of \cite{GM1}. This gives us our claim. 
\end{proof} 


	As slight additional clarification, 
	one recalls the definition of $p_t(x)$ and Gaussian bounds on the heat kernel:
	\beq\label{ineq:heat_kernel}
	c_1 (4\pi t)^{-n/2} e^{-c_2\frac{\text{dist} (x, y)^2}{4t}} \leq	K(t, x, y) \leq c_3 (4\pi t)^{-n/2} e^{-c_4\frac{\text{dist} (x, y)^2}{4t}}
	\eeq
	on the heat kernel $K(t, x, y)$ on any Riemannian manifold. One can of course look at (\ref{ineq:heat_kernel}) in the four regimes given by near/off diagonal and long/short time, and there exist bounds which work very generally (see \cite{CGT}, \cite{LY} and \cite{Gr} as a general reference). Here we are interested only in the regime of short time near diagonal bounds, and here one can take 
	$c_1 $ and $c_3$ as arbitrarily close to $1$ as one wants, and $c_2 = c_4 = 1$. This follows from the Minakshisundaram-Pliejel estimates
	$$
	\left| K(t, x, y)- \frac{e^{-\frac{\text{dist}^2(x,y)}{4t}}}{(4\pi t)^{n/2}}\sum_{k=0}^{N} u_k(x,y)t^k \right| \leq C_{N}t^{N+1-n/2}  e^{-\frac{d^2(x,y)}{4t}}.
	$$
	For more details, consult \cite{Lu} (see also Sections 3.2, 3.3 of \cite{R}). Alternatively, this will also follow from the large deviation formula of Varadhan in (\ref{eq:Var_lar_dev}). Additionally, in the presence of lower Ricci curvature bounds $\Ric \geq -\kappa, \kappa \geq 0$, one can take $c_2 = c_2(\kappa)$, and this actually works in all regimes (for the explicit formula, see \cite{B}). 
	Now we discuss the upper bound in (\ref{ineq:heat_kernel}). For integral kernels of  $f(-\Delta)$ for a wide class of functions $f$ (not only $f_t(x) = e^{-tx}$, which corresponds to the heat kernel), one can make the constant $c_3$ a universal constant independent of geometry; see Theorem 1.4 of \cite{CGT} for an even more general statement. The above, in conjunction with the definition of Martin capacity (which involves an integration of the cut-off Green's function), and the volume comparison of balls that depend on the sectional curvature bounds, give us the explicit dependence of constants as stated above in (\ref{ineq:heat_upper}). 
   In the interest of completeness, we refer the reader to Lemma 3.4, pp 210 of \cite{SY}, which gives an explicit Taylor series expansion of the volume form in geodesic normal coordinates. To wit, the first few terms of the formula reads 
   $$
   \det g_{ij} = 1 - \frac{1}{3}R_{ij}x^ix^j -\frac{1}{6}\nabla_kR_{ij}x^ix^jx^k + O(r^4),
   $$
   from which one can easily write down a Taylor series expansion of the volume form $\sqrt{\det g_{ij}}$ involving the curvature terms and their derivatives.

Now we come back to the proof of Theorem \ref{thm:conc_L1}. The bound 
(\ref{ineq:heat_upper}) implies
$$
\int_{M \setminus T_{r}} |p_t(x)\varphi_\lambda(x)| dx \leq c_1 
\Theta_n\left(r^2/t\right) \| \varphi_\lambda 
\|_{L^1(M \setminus T_{r})},
$$
which gives that
\begin{align*}
\| \varphi_\lambda \|_{L^1(T_r)} & \geq \int_{T_r} |p_t(x)\varphi_\lambda(x)| dx = \| p_t\varphi_\lambda\|_{L^1(M)} - \| p_t\varphi_\lambda\|_{L^1(M \setminus T_r)}\\
& \geq (1 - e^{-t\lambda}) \| \varphi_\lambda\|_{L^1(M)} - c_1
\Theta_n\left(r^2/t\right)\| \varphi_\lambda\|_{L^1(M \setminus T_r)}\\
& \geq \left(1 - e^{-t\lambda} - c_1
\Theta_n\left(r^2/t\right)\right)\| \varphi_\lambda\|_{L^1(M)}.
\end{align*}
This proves (\ref{ineq:L1_conc}).

\subsection{Proof of Theorem \ref{thm:conc_L1_n=2}}
We first discuss the case of $p  = 1$. By working separately on each nodal domain, and adding over all the nodal domains, We already know that 
\begin{align*}
    \int_M p_t(x) |\varphi(x)|\; dx & = \left( 1 - e^{-t\lambda}\right) \int_M |\varphi(x)| \; dx = \int_{T_r} p_t(x) |\varphi(x)|\; dx + \int_{M \setminus T_r} p_t(x) |\varphi(x)|\; dx,
\end{align*}
which implies that 
\beq
\left( 1 - e^{-t\lambda}\right) \int_M |\varphi(x)| \; dx \geq \int_{T_r} p_t(x) |\varphi(x)|\; dx.
\eeq
Now it is clear that 
if we had a suitable lower bound on $p_t(x)$ in terms of $\text{dist} (x, \pa \Omega)$ when $x$ is close 
to the boundary $\pa \Omega$, we would be through. However, such a statement cannot be expected to hold in dimensions $n \geq 3$. As an example, one can imagine $x$ being close to a ``sharp spike'' of very low capacity. For a 
Brownian particle starting at a point inside the nodal domain which is wavelength-near from the ``tip'' of one such spike, 
the probability of striking the nodal set is still negligible. In dimension $2$, we will argue that such spikes cannot exist, and the probability of a Brownian particle hitting any curve is bounded from below depending only on the distance of the curve and the starting point of the Brownian particle. 
The same phenomenon is at the heart of why in dimension $n = 2$, domains are proven to have wavelength 
inner radius (see \cite{H}), but in higher dimensions, one cannot make such a conclusion without allowing for a volume error. 
For more details, refer to \cite{L}, \cite{MS} and \cite{GM1} (in the special setting of real analytic manifolds, improvements 
can of course be made; see \cite{G}, which combines arguments from \cite{JM} and \cite{GM1} and uses rather strongly the 
analyticity of the domain).

Now we formalize the above heuristic: consider a piece-wise smooth domain $\Omega$ of 
dimension $n = 2$, where we wish to prove that if we set $t = r_0^2\lambda^{-1}$, we can localize the integral
$$
\int_\Omega |p_t(x) \varphi_\lambda(x)| dx \gtrsim \int_{T_{r_0\lambda^{-1/2}}} |\varphi_\lambda (x)| dx.$$
From what has gone above, it suffices to prove a quantitative estimate which says that if $x$ close to $\pa \Omega$, then $p_t(x)$ is high. Suppose we have $\text{dist} (x, \pa \Omega) < kr$, where $k$ is a constant. It is known that there exists such an $r$ such that for any eigenvalue $\lambda$ and disk $B \subset M$ of radius $\leq r\lambda^{-1/2}$, there exists a $K$-quasiconformal map $h : B(x, r) \to D \subset \CC$, where $D$ is the unit disk in the plane such that $x$ is mapped to the origin (for more details, see Theorem 3.2 of \cite{Ma2}, and also \cite{N}, \cite{NPS}). 
 Now, define $E : = D \setminus 
h(\Omega \cap B(x, kr))$, and assume that $\psi_E(t, 0) \leq \epsilon \leq \delta$, where $\delta$ is a small enough constant, assumed without loss of generality to be $< 1/2$. We will now investigate the implication of the above assumption on $
w(0, E) $, the harmonic measure of the set $E$ with a pole at the origin. Adjusting the ratio $k^2r^2/t$ suitably depending on $\delta$, we can arrange that the probability of a Brownian particle starting at at the origin to hit the boundary $\pa D$, and hence $h (\pa B(x, kr))$ within time $t$ is at least $1 - \delta$. Setting $\psi_E (\infty, 0)$ to be the probability of the Brownian particle starting at $x$ to hit $E$ after time $t$, we have that 
$$
w(0, E) = \psi_E(t, 0) + \psi_E(\infty, 0) \leq \epsilon + \delta \leq 2\delta.$$
On the other hand, by the Beurling-Nevanlinna theorem  (see \cite{A}, Section $3$-$3$), 
$$
w(0, E) \geq 1 - c\sqrt{\text{dist} (0, E)}, $$
which shows that 
$$
\text{dist} (0, E) \gtrsim (1 - 2\delta)^2.$$
This proves our contention.\\
\\

\begin{center}
	\includegraphics[width=80mm,scale=0.8]{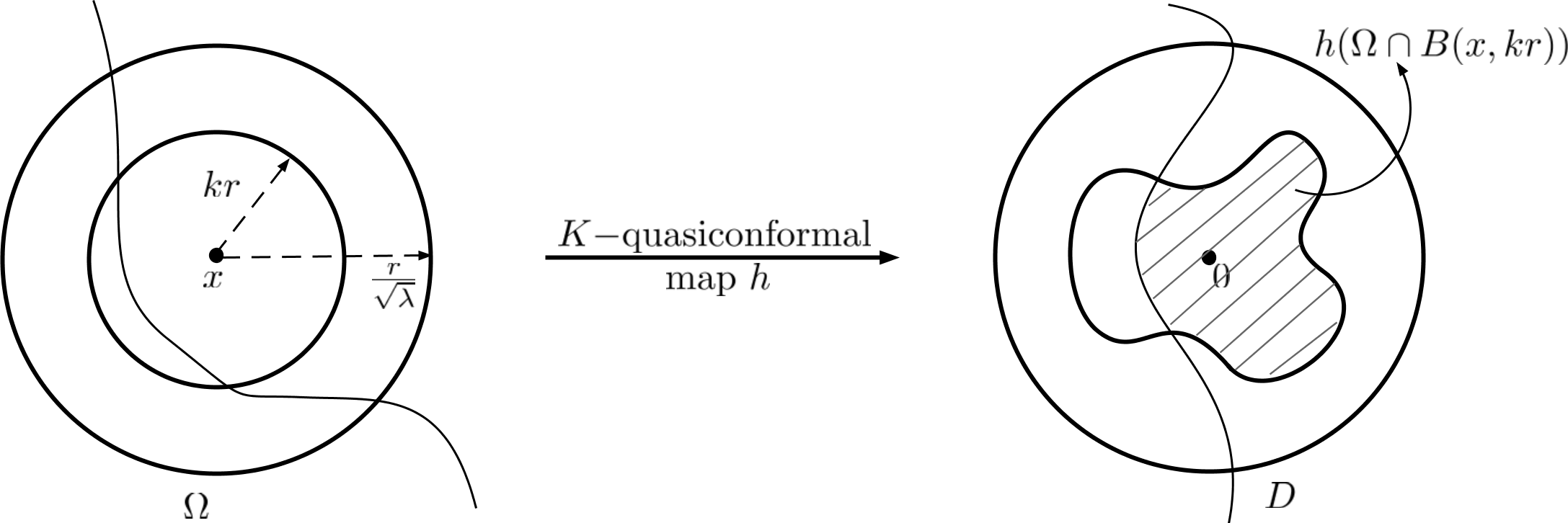}\label{diag3}
\end{center}

\vspace{7mm}

Now we quickly indicate how to prove (\ref{eq:L1_conc}) for $p > 1$. 
Define $\Phi(t, x) := e^{-p\lambda t}\varphi_\lambda^p(x)$. It can be easily checked that $\Phi$ satisfies 
\beq\label{eq:Phi_satisfies}
(\pa_t - \Delta)\Phi = - p(p - 1)e^{-p\lambda t}\varphi_\lambda^{p - 2}|\nabla \varphi_\lambda|^2.
\eeq
Observe that to repeat the previous computation for the $L^1$-norm, we need a proper expression for $e^{t\Delta}\varphi_\lambda^p$. So let us set 
$$
e^{t\Delta}\varphi_\lambda^p = \Phi (t, x) + \Psi (t, x).
$$
Since $\Phi + \Psi$ solves the heat equation with initial condition $\varphi_\lambda^p(x)$, from (\ref{eq:Phi_satisfies}) one can see that $\Psi$ solves
$$
(\pa_t - \Delta) \Psi = p(p - 1)e^{-p\lambda t}\varphi_\lambda^{p - 2}|\nabla \varphi_\lambda|^2
$$
with initial condition $\Psi (0, x) = 0$. By Duhamel's formula, 
$$
\Psi (t, x) =  \int_0^t \int_M K(t - s, x, y) e^{-p\lambda s}p(p- 1)\varphi_\lambda^{p - 2}(y)|\nabla\varphi_\lambda (y)|^2 \;dy \;ds.
$$
So, using (\ref{eq:pt_f}) and adding over nodal domains as before, we have that 
\begin{align*}
\int_M p_t(x)|\varphi_\lambda(x)|^p dx & = \int_M (1 - e^{t\Delta})|\varphi_\lambda(x)|^p dx\\
& = \int_M |\varphi_\lambda(x)|^p dx - \int_M \Phi (t, x) dx - \int_M \Psi(t, x) dx\\ 
& = \int_M |\varphi_\lambda(x)|^p dx - e^{-p\lambda t}\int_M |\varphi_\lambda (x)|^p dx - \int_M \Psi(t, x) dx \\
& = (1 - e^{-p\lambda t})\int_M |\varphi_\lambda(x)|^p dx - \int_M \Psi(t, x) dx\\
& \leq (1 - e^{-p\lambda t})\int_M |\varphi_\lambda(x)|^p dx.
\end{align*} 

From the expression derived by Duhamel's formula directly, or directly from the parabolic maximum principle, it is clear that $\Psi (t, x) \geq 0$. The rest of the proof is exactly analogous to the $L^1$ case and we spare the reader the repetition.

\section{Avoided crossings: proof of Theorem \ref{thm:thin_nod}}\label{sec:av_cros}

The following result was first proved in \cite{Ma1}, and can be interpreted as a quantitative version of avoided crossings.

\begin{theorem}[Theorem 1.9, \cite{Ma1}]\label{thm:Man}
	Let $(M, g)$ be a closed smooth surface, $B \subset M$ be an arbitrary ball of radius $R$, and $\Omega_\lambda$ be a connected component of $\{\varphi_\lambda \neq 0\} \cap B$. If $\Omega_\lambda \cap \frac{1}{2}B \neq \emptyset$, then 
	$$
	\frac{|\Omega_\lambda|}{|B|} \gtrsim \frac{1}{\lambda R' (\log \lambda)^{1/2}},
	$$
	where $R' = \max \{R, \frac{1}{\sqrt{\lambda}}\}$.
\end{theorem}

This is a variant of the phenomenon of rapid growth in narrow domains, proven by elliptic PDE methods. As in the rest of the paper, our method is essentially heat theoretic/probabilistic, and we start by modifying the heat equation approach to address the problem of avoided crossings. We mention \cite{St} as a forerunner of such an approach, and in 
the process of our proof of Theorem \ref{thm:thin_nod}, we will also outline some details not explicitly mentioned in \cite{St}. 
We also recommend the reader to consult Section 1 of \cite{MSG}
to get some physical perspective on these kinds of problems.


We start by choosing a nodal domain $\Omega \subset M$. For the ensuing computations, it will be convenient to normalize $\| \varphi_\lambda\|_{L^\infty(\Omega)} = 1$. Now, as mentioned before, denote by $\locnod := \Omega \cap B_R$ the local nodal domain. 
Start by covering $\locnod^s$ with $N$ cubes $Q_1, Q_2,..., Q_N$ with side length $
r$, where $r \sim 
\lambda^{-\alpha}, \alpha > 1/2$. 
Also, extend the covering to cover $ \locnod$ with $N + \tilde{N}$ cubes of side length $\beta r$, $Q_1,..., Q_N, Q_{N + 1},..., Q_{N + \tilde{N}}$. 
We posit that we have a comparability $N \sim \tilde{N}$ upto constants dependent on $s$ and the geometry of $(M, g)$. 
Recall that we are interested in the diameter of the local nodal domain $\locnod^s$. 

\includegraphics[width=130mm,scale=2]{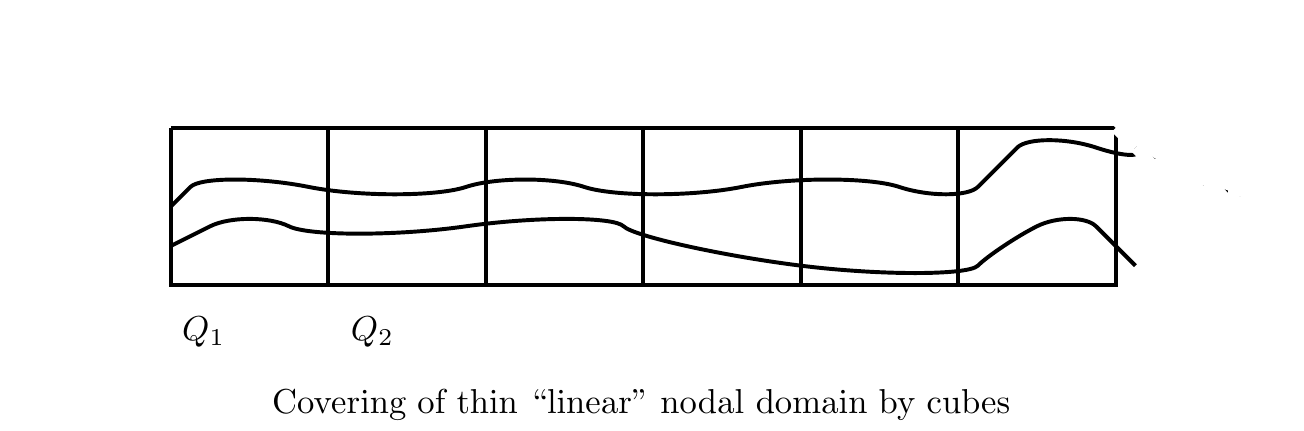}\label{diag1}
\\

\begin{proof} As before, we start by considering the ``insulated boundary'' diffusion process for $\Omega$ given by (\ref{eq:Heat_Flow}), and will use the Feynman-Kac formulation given by (\ref{eq:F-K}). Observe that a Brownian motion started at a fixed point in the set $Q_i \cap \locnod^s$ at time $t$ can impact on the boundary $\pa\locnod \setminus \pa B_R 
$ with probability $p_{ib}$, can end up in any of the other sets  $Q_j \cap \locnod$  with probability $p_{ij}$, or exit the set through the ``ends'', that is, land up inside $\Omega \setminus \locnod$ with probability $p_{ie}$, which gives
\beq\label{eq:prob_add_1}
p_{ib} + \sum_{j = 1}^{\tilde{N}} p_{ij} + p_{ie} = 1.
\eeq
We will soon stipulate that we are looking at the Feynman-Kac formula at distance scales $\sim r$ and time scales $\sim r^2$. Due to both sided Gaussian bounds on the heat kernel, as discussed before, we see that 
\beq\label{ineq:pij_bound}
p_{ij} \lesssim e^{-c_1\frac{\text{dist}(Q_i, Q_j)^2}{r^2}}. 
\eeq
Also for any $i \leq N$ 
\beq\label{ineq:pia_small}
 p_{ie} \lesssim e^{-c_2N},
  \eeq
 where the constant $c$ in the last inequality depends on $s$. It is clear that we cannot get (\ref{ineq:pia_small}) unless we restrict to a smaller concentric ball $B_{sR}$ inside $B_R$.

Now, by selecting a starting point for the Brownian motion in $Q_i$ where $\sup_{Q_i} |\varphi_\lambda|$ is realized, the Feynman-Kac formula fitted to the Dirichlet heat flow on $\Omega$ (as before) yields that 
\begin{align}\label{ineq:F_K_consequence}
e^{-\lambda t}\sup_{Q_i}|\varphi_\lambda| & = \int_{\omega(t) \in \Omega \setminus \cup_i Q_i} \varphi_\lambda (\omega (t)) \phi_\Omega(\omega, t) d\omega  + \sum_{j = 1}^{\tilde{N}} \int_{\omega(t) \in Q_{j}} \varphi_\lambda (\omega (t)) \phi_\Omega(\omega, t) d\omega \nonumber\\
& \leq p_{ie} \sup_{\Omega} |\varphi_\lambda| + \sum_{j = 1}^{\tilde{N}} p_{ij} \sup_{Q_j}|\varphi_\lambda|.
\end{align}
 Now, we make 
 the following important assumption, whose validity we will discuss below: for the time being, we assume that {\em $p_{ib}$ is uniformly bounded below away from $0$.} 
Using this with (\ref{eq:prob_add_1}), the uniform bound on $p_{ib}$ tells us that $\sum_{j = 1}^{\tilde{N}} p_{ij}$ is uniformly bounded above away from $1$, let's say $\sum_{j = 1}^{\tilde{N}} p_{ij} \leq \nu < 1$. With that in place, 
choose an initial starting cube $Q_{i_1}$ such that the center of $Q_{i_1}$ is inside $B_{\frac{sR}{2}}(x)$. Now, (\ref{ineq:pia_small}) above tells us that $p_{i_1e}$ is uniformly small, and in fact negligible, because $r$ is very small compared to $R$. 

Now, let us choose a constant $c_4$ much smaller than $c_1$. 
If there exists a cube $Q_j$ whose center 
is inside $B_{\frac{3sR}{4}}(x)$ such that 
$$
\sup_{Q_j}|\varphi_\lambda| \geq c_3e^{c_4\frac{\text{dist}(Q_{i_1}, Q_j)^2}{r^2}}\sup_{Q_{i_1}}|\varphi_\lambda|, 
$$
we just accumulate some growth from $Q_{i_1}$ to $Q_j$, relabel the cube $Q_j$ as $Q_{i_2}$ and start again with the cube $Q_{i_2}$, till we reach the point where the above procedure cannot be iterated. Similarly, if there exists a cube $Q_j$ whose center 
is inside $B_{\frac{3sR}{4}}(x)$ such that 
$$
\sup_{Q_j}|\varphi_\lambda| \leq c_3e^{-c_4\frac{\text{dist}(Q_{i_1}, Q_j)^2}{r^2}}\sup_{Q_{i_1}}|\varphi_\lambda|, 
$$
we just accumulate some growth from $Q_j$ to $Q_{i_1}$, relabel the cubes $Q_{i_1}$ as $Q_{i_2}$ and $Q_j$ as $Q_{i_1}$ and start again with the cube $Q_{i_1}$, till we reach the point where the above procedure cannot be iterated.

So without loss of generality, one can assume that one has 
$$
c_3e^{-c_4\frac{\text{dist}(Q_{i_1}, Q_j)^2}{r^2}}\sup_{Q_{i_1}}|\varphi_\lambda| \leq 
\sup_{Q_j}|\varphi_\lambda| \leq c_3e^{c_4\frac{\text{dist}(Q_{i_1}, Q_j)^2}{r^2}}\sup_{Q_{i_1}}|\varphi_\lambda|.
$$
With this assumption, we now bring in the Gaussian decay of $p_{{i_1}k}$ as in (\ref{ineq:pij_bound}) and the fact that $\sum_{j = 1}^{\tilde{N}} p_{ij} \leq \nu < 1$. We choose $t = t_0 r^2$ with a suitable constant $t_0$ depending on $\nu$. Now, using (\ref{ineq:F_K_consequence}), we can pick a cube $Q_{i_2}$ 
such that the center of $Q_{i_2}$ is inside $B_{\frac{3sR}{4}}(x)$ and which satisfies
\beq\label{ineq:Stefan_magic}
\sup_{Q_{i_1}}|\varphi_\lambda| 
\lesssim c_5 e^{-c_6N} + 
e^{-c_7\frac{\text{dist}(Q_{i_1}, Q_{i_2})^2}{r^2}}\sup_{Q_{i_2}}|\varphi_\lambda|,
\eeq
where $c_7 < c_1$ is a positive constant depending only on $(M, g)$. 
Now we start with the cube $Q_{i_2}$ and reiterate the last  process. We stop the iteration when 
the center of the cube $Q_{i_k}$ falls outside $B_{\frac{3sR}{4}}(x)$. The main idea is to iterate the procedure enough number of times so that one accumulates enough growth to contradict the Donnelly-Fefferman growth bounds. 
If the iteration ends after the $k$th step, 
the worst possible case is that  $k \sim N$, whence we have at the end 
\beq
 \sup_{Q_{i_1}} |\varphi_\lambda | 
 \lesssim k e^{-c_6N} + c^N \sup_{Q_{i_e}}|\varphi_\lambda|,
\eeq
where $c < 1$ is a positive constant independent of $k$. On the other hand, the Donnelly-Fefferman growth bound (see \cite{DF}, or Theorem 3.2 of \cite{Z}) combined with the well-known finite chain of balls argument says that 
\beq
\sup_{Q_{i_1}} |\varphi_\lambda | \gtrsim e^{\sqrt{\lambda}\log r}.
\eeq
Putting the above estimates together shows that 
\beq
N \lesssim -\sqrt{\lambda}(\log r),
\eeq
which proves our contention.

	

The only issue that remains to be addressed is the uniform boundedness (below away from $0$) of $p_{ib}$. Observe that in dimension $n = 2$, 
this follows from the heuristic that any curve has positive capacity. A formal proof will follow the same argument that has been formalized near the end of Section \ref{sec:mass_conc} involving the quasiconformal map, Beurling-Nevanlinna theorem and harmonic measure. 

\end{proof}

\section{Proof of Theorem \ref{thm:decay_narrow_branch} and allied remarks}
Consider again the insulated boundary formulation (\ref{eq:Heat_Flow}) and (\ref{eq:F-K}). Starting the Brownian motion at $x \in \Omega$, we have the following bounds:
\begin{align}\label{stein1}
\Phi (t, x) & = e^{-\lambda t}\varphi (x) = \mathbb{E}_x(\varphi(\omega(t))\phi_{\Omega}(\omega, t))\\
& \leq \V \varphi\V_{L^\infty(\Omega)}\mathbb{E}_x(\phi_{\Omega}(\omega, t)) = \V \varphi\V_{L^\infty(\Omega)}(1 - p_t(x))\nonumber.
\end{align}
Setting $t = 1/\tilde{\lambda}$, 
we see that 
\begin{equation*}
\eta \leq e^{\frac{\lambda}{\tilde{\lambda}}}( 1 - \psi_{\RR^n \setminus \Omega}(\tilde{\lambda}^{-1}, x)). \end{equation*}
Rewriting, we have that
\beq\label{ineq:exp_est}
\varphi (x) \leq 
e^{\frac{\lambda}{\tilde{\lambda}}}\left( 1 -  \psi_{\RR^n \setminus \Omega}\left(\tilde{\lambda}^{-1}, x\right)\right) \| \varphi \|_{L^\infty(\Omega)}.
\eeq

This quantifies the phenomenon that if $\lambda$ is small in comparison to $\mu$, then the eigenfunction corresponding to $\lambda$ cannot ``squeeze'' through a tunnel of sub-wavelength width. Here, our estimate (\ref{ineq:exp_est}) applies in all dimensions.

We take the space to remark that results in a similar vein have been investigated before for unbounded domains on the plane. For example, one can take ``horn-shaped'' domains, as considered in \cite{DS} and \cite{BD}, which are defined as follows: let $\theta : (0, \infty) \to [0, 1]$ be a continuous function
and $\Omega_\theta = \{(x_1, x_2, ..., x_n) : \sum_{j = 2}^{n} x_j^2 \leq \theta^2\}$. We just assume that $\theta $, which is a function of $x_1$ alone, is vanishing at infinity. See \cite{DS} and Theorem $1$ of \cite{BD} for exponential upper bounds
on the eigenfunctions $\varphi_\lambda$ corresponding to the Dirichlet Laplacian on $\Omega_\theta \subset \RR^2$. As a model example, let us quote the following result from \cite{BD}. Under the assumption that $n = 2$, $\theta^{'}$
bounded and $C^1$, $\theta^{'} \leq 0$ for large $x$, and $\theta^{'}/\theta \to 0$ at infinity, \cite{BD} proved the following bounds
$$
e^{-\frac{c_1 x}{\theta}} \lesssim \varphi_\lambda (x, 0) \lesssim e^{-\frac{c_2 x}{\theta}},
$$
where $c_1, c_2$ are constants. Such exponential decay 
results have been investigated before also, particularly starting with the seminal work of Agmon (\cite{Ag}), but as \cite{BD} points out, 
variants based on \cite{Ag} do not seem to be directly applicable to regions of $\RR^n$ because of lack of completeness of the Agmon metric (for further details, consult \cite{HiSi}). 

\section{Appendix: Some ideas on interaction of level and sub-level sets} We now note down some ideas as to how our heat-theoretic methods could be used to derive information about sub-level sets of $\varphi_\lambda$. We hope such ideas will have future applications.

Let $M$ be a closed Riemannian manifold with smooth metric of dimension $n \geq 2$. For the sake of convenience, let us normalize $\| \varphi_\lambda\|_{L^\infty (M)} = 1$. For any $S \subset M$, let $S^c$ denote the set $M \setminus S$. We will also denote the sub-level set $S_r := \{ x \in M : |\varphi_\lambda (x)| \leq r\}$, and the level set $L_r := \{ x \in M : |\varphi_\lambda (x)| = r\}$. 

First, we comment on some heat interpretations of variants of results in \cite{FS}. Given $\eta \in (0, 1)$, we have 
\begin{align*}
\int_\Omega p_t(x) \varphi_\lambda (x) dx & = \int_{T_\delta \cap \Omega} p_t (x)\varphi_\lambda (x) dx + \int_{T_\delta^c \cap \Omega} p_t (x)\varphi_\lambda (x) dx \\
& = \int_{T_\delta \cap S_\eta \cap \Omega} p_t (x)\varphi_\lambda (x) dx + \int_{T_\delta \cap S_\eta^c \cap \Omega} p_t (x)\varphi_\lambda (x) dx \\ 
& + \int_{T_\delta^c \cap S_\eta \cap \Omega} p_t (x)\varphi_\lambda (x) dx + \int_{T_\delta^c \cap S_\eta^c \cap \Omega} p_t (x)\varphi_\lambda (x) dx.
\end{align*}
On $T_\delta^c$ as discussed before, $p_t(x) \lesssim 
\Theta_n\left(\delta^2/t\right)$. Also by an application of the Feynman-Kac formula (\ref{eq:F-K}), we have that for $x \in S_\eta^c$, 
$$
e^{-t\lambda}\varphi_\lambda (x) \leq \mathbb{E}(\phi_{\Omega}(\omega, t))) =  (1 - p_t(x)), $$
which gives that $p_t(x) < (1 - \eta e^{-\lambda t})$. Plugging these in, 
we get that 
\begin{align*}
(1 - e^{-\lambda t})\| \varphi_\lambda\|_{L^1(\Omega)}  & \lesssim \eta\left|T_\delta \cap S_\eta \cap \Omega\right| + (1 - \eta e^{-\lambda t})|T_\delta \cap S_\eta^c \cap \Omega| \\
& + 
\Theta_n\left(\delta^2/t\right)\eta |T_\delta^c \cap S_\eta \cap \Omega| + 
\min\left\{ \Theta_n\left(\delta^2/t\right), 1 - \eta e^{-\lambda t}\right\}|T_\delta^c \cap S^c_\eta \cap \Omega|.
\end{align*}
Given $\varepsilon$ (possibly depending on $\lambda$), 
$\delta = \delta_0/\sqrt{\lambda}$, one can choose $\eta$ (depending on $\lambda$) close to $1$, and $t = t_0/\lambda$ such that 
$\Theta_n\left(\delta^2/t\right) = \Theta_n\left( \delta^2_0/t_0\right)$ and $ 1 - \eta e^{-\lambda t} \leq \frac{\varepsilon}{3|M|}$. Finally, adding over all the nodal domains, this yields that 
\beq\label{ineq:vol_tub_sub}
|T_\delta \cap S_\eta| + \varepsilon \gtrsim_{\delta,\varepsilon} 
\| \varphi_\lambda\|_{L^1(M)}.
\eeq

It is a rather difficult problem to to give a lower bound on the volume of sub-level sets of eigenfunctions $\varphi_\lambda$. This problem is intrinsically tied with the investigation of level sets and tubular neighbourhoods of Laplace eigenfunctions. Recently, \cite{FS} has established the following: 
\begin{theorem}\label{thm:comp_tub_sub}
	Let us normalize the Laplace eigenfunctions $\varphi_\lambda$ to have  $\|\varphi_\lambda \|_{L^2} = 1$ for all eigenvalues $\lambda$. Then given $\epsilon, \delta > 0$ and $r \leq \frac{\delta}{\sqrt{\lambda}}$, we have the following lower bound
	\beq
	| S_{r\epsilon^{-1}} | \gtrsim_\delta r\lambda^{1/2 - \delta} - \epsilon.
	\eeq
\end{theorem}
\noindent The main technical lemma is to establish a relation between the tubular neighbourhood and sublevel sets. More formally, Lemma 3.3 of \cite{FS} gives:
\begin{lemma}\label{Lem:FS}
	Given $\epsilon > 0$, we have $E \subset M$ satisfying $|E| \leq \epsilon$, such that  
	\beq\label{ineq:Lem_FS}
	\left(M \setminus E\right) \cap T_{r\lambda^{-1/2}} \subset \left(M \setminus E\right) \cap S_{\frac{r \|\varphi_\lambda\|_{L^2}}{\epsilon}}.
	\eeq 	
\end{lemma}
\noindent The main idea is to use a result of Colding-Minicozzi (see \cite{CM}) which gives appropriate pointwise upper bounds on $|\nabla \varphi_\lambda|$ outside an exceptional set of small measure. In \cite{GM2}, the present authors derived the following lower bounds for the volume of a tubular neighbourhood:
\beq\label{ineq:Tub_nbd_bd}
|T_\delta| \gtrsim_\rho \lambda^{1/2 - \rho} \delta.
\eeq
Lemma (\ref{Lem:FS}), combined with (\ref{ineq:Tub_nbd_bd}) above then gives lower bounds on the sublevel sets of $\varphi_\lambda$. (\ref{ineq:vol_tub_sub}) should be seen as complementary to Theorem \ref{thm:comp_tub_sub} above.

Now we take the space to make some comments about how interactions between different level sets could be captured by heat theoretic methods.

\begin{proposition}\label{prop:level_sets_interaction}
Consider a domain $\Omega$ (with boundary) sitting inside a closed smooth manifold $M$ of dimension $n$, and let $\varphi$ be the ground state Dirichlet eigenfunction of $\Omega$ corresponding to the eigenvalue $\lambda$. 
Normalize $\| \varphi\|_{L^\infty(\Omega)} = 1$ and fix two numbers $0 < \mu < \eta \leq 1$. Now,  choose a point $x$ such that $\varphi(x) = \mu$. Then, for any positive number $\tau$, we have that 
$$
\tilde{\psi}_{S^c_\eta}(\tau, x) \leq \frac{\mu}{\eta}\left( \frac{1 - e^{-\lambda\tau}}{\lambda}\right),
$$
where $\tilde{\psi}_{S^c_\eta} (\tau, x)$ denotes the probability that a Brownian particle starting at $x$ strikes $S^c_\eta$ within time $\tau$ remaining within $\Omega$ the whole time. 
\end{proposition}
Heuristically, this gives a sort of a dichotomy: if the level sets $L_\eta$ and $L_\mu$ are close to each other, then $L_\mu$ must be close to the boundary $\pa\Omega$, so that a lot of Brownian particles get killed on impact with  $\pa \Omega$ before striking $L_\eta$.  

\begin{proof}
    The proof is simple. By the Feynman-Kac formula (\ref{eq:F-K}) above, we have that
    \begin{align*}
        e^{-\lambda t}\varphi(x) & = \int_{\omega(t) \in S_\eta^c} \varphi  (\omega (t)) \phi_\Omega(\omega, t) \; d\omega  + \int_{\omega(t) \in S_\eta} \varphi  (\omega (t)) \phi_\Omega(\omega, t)\; d\omega\\
        & \geq \eta \int_{\omega(t) \in S_\eta^c}  \phi_\Omega(\omega, t) \; d\omega.
    \end{align*}
Integrating the above inequality, we get that 
\begin{align*}
\mu \left( \frac{1 - e^{-\lambda\tau}}{\lambda}\right) & \geq \eta \int_{t = 0}^\tau \int_{\omega(t) \in S_\eta^c}  \phi_\Omega(\omega, t) \; d\omega \\
& \geq \eta\; \tilde{\psi}_{S^c_\eta} (\tau, x).
\end{align*}
\end{proof}

\vspace{4 mm}

\section{Acknowledgements} The work was started when the second author was at the Technion, Israel supported by the Israeli Higher Council and the research of the second author leading to these results is part of a project that has received funding from the European Research Council (ERC) under the European Union's Horizon 2020 research and innovation programme (grant agreement No 637851). The work was continued when the second author was at the Courant Institute of Mathematical Sciences hosted by Fanghua Lin, the Max Planck Institute for Mathematics, Bonn and finished when he was at IIT Bombay. He wishes to deeply thank them all. The first author gratefully acknowledges the Max Planck Institute for Mathematics, Bonn and the Fraunhofer Institute, IAIS, for providing ideal working conditions. The first author's research is also within the scope of the ML2R project and Research Center, ML at Fraunhofer, IAIS. The authors also thank Steve Zelditch, Michiel van den Berg, Gopal Krishna Srinivasan, Stefan Steinerberger and Rajat Subhra Hazra for helpful conversations. In particular the authors are deeply grateful to Emanuel Milman for reading a draft version in detail and pointing out several corrections and improvements. 

\end{document}